\documentclass[a4paper,11pt]{amsart}
\usepackage{graphicx}
\usepackage{mathptmx}
\usepackage{amsmath}
\usepackage{amssymb}
\usepackage{enumitem}
\usepackage{xcolor}
\usepackage{dsfont}
\numberwithin{equation}{section}
\newtheorem{theorem}{Theorem}
\newtheorem{lemma}{Lemma}
\newtheorem{corollary}[theorem]{Corollary}

\newtheorem{definition}[theorem]{Definition}
\newtheorem{remark}{Remark}

\begin{document}

\title{Generalized Fubini transform with two variables}

\author{Madjid Sebaoui}
\address[M. Sebaoui]{Dr Yahia Far\`{e}s University of M\'{e}d\'{e}a 26000 M\'{e}d\'{e}a, Algeria}
\email{msebaoui@gmail.com, sebaoui.madjid@univ-medea.dz}

\author{Diffalah Laissaoui}
\address[D. Laissaoui]{Dr Yahia Far\`{e}s University of M\'{e}d\'{e}a 26000 M\'{e}d\'{e}a, Algeria}
\email{laissaoui.diffalah74@gmail.com}

\author{Ghania Guettai}
\address[G. Guettai]{Dr Yahia Far\`{e}s University of M\'{e}d\'{e}a 26000 M\'{e}d\'{e}a, Algeria}
\email{guettai78@yahoo.fr}

\author{Mourad Rahmani}
\address[M. Rahmani]{USTHB, Faculty of Mathematics, P.O. Box 32, El Alia 16111, Algiers, Algeria}
\email{mourad.rahmani@gmail.com, mrahmani@usthb.dz}

\begin{abstract}
In the present paper, we define the generalized  Kwang-Wu Chen matrix. Basic properties of this generalization, such as explicit formulas and generating functions are presented. Moreover, we focus on a new class of generalized Fubini polynomials. Then we discuss their relationship with other polynomials such as Fubini, Bell, Eulerian and Frobenius-Euler polynomials.  We have also investigated some basic properties related to the degenerate generalized Fubini polynomials.
\end{abstract}
\maketitle

\section{Introduction}
The n$th$ Bernoulli numbers $B_{n}$  are defined by the generating
function
\begin{equation}
\frac{t}{e^{t}-1}=
{\displaystyle\sum\limits_{n\geq0}}
B_{n}\frac{t^{n}}{n!},\ \ \left\vert t\right\vert <2\pi.
\end{equation}
The rational numbers $B_{0}=1$, $B_{1}=-1/2$, $B_{2}=1/6$, $B_{3}=0$, $B_{4}=-1/30$ and
$B_{2n+1}=0$ for $n>0$, have many beautiful properties. The most basic
recurrence relation is
\begin{equation}
{\displaystyle\sum\limits_{k=0}^{n}}
\dbinom{n+1}{k}B_{k}=0. \label{A2}%
\end{equation}
In 2001, Kwang-Wu Chen \cite{chen2001algorithms}  gave an algorithm for computing Bernoulli numbers, with
\begin{equation}
a_{0,m}=\frac{1}{m+1};\quad a_{n+1,m}  =-\left(  m+1\right)  a_{n,m+1}+ma_{n,m}.\label{rechen}
\end{equation}

The primary purpose of this paper is to extend the Fubini transform for generalizing Fubini polynomials and studying its properties.
We first generalize (\ref{rechen}). The idea is to construct an infinite matrix $\mathcal{M}:=\left(  a_{n,m}\right)  _{n,m\geq0}$ in which the first row $a_{0,m}:=\alpha_{m}$ of the matrix is the initial sequence and the first column
$a_{n,0}:=\beta_{n}$ is the final sequence.
More precisely, for nonzero complex numbers $x$ and $y$,
we propose to study the following three-term recurrence relation
\begin{equation}
a_{n+1,m}(x,y)  =x\left(  m+1\right)  a_{n,m+1}(x,y)+yma_{n,m}(x,y).\label{rec1}
\end{equation}
By setting $x=-1$ and $y = 1$ in (\ref{rec1}), we get (\ref{rechen}).
More directly, we propose to generalize the Fubini transformation.\\ \\
\indent
The Fubini transform
 of a sequence $\left(  \alpha
_{n}\right)  _{n\geq0}$ is the sequence $\left(  \beta_{n}\right)  _{n\geq0}$
given by
\[
\beta_{n}=
{\displaystyle\sum\limits_{k=0}^{n}}
k!\genfrac{\{}{\}}{0pt}{}{n}{k}
t^{k}\alpha_{k}
\]
and the inverse transform is
\[
\alpha_{n}=
\frac{1}{n! t^{n}}
{\displaystyle\sum\limits_{k=0}^{n}}
s\left(  n,k\right)
\beta_{k\text{ \ }}.
\]

\section{Definitions and notation}
In this section, we introduce some definitions and notations which are useful in the
rest of the paper. Following the usual notations \cite{Knuth}. \\
\indent The falling factorial
$x^{\underline{n}}$ $\left(  x\in\mathbb{C}\right)  $ is defined by
$$x^{\underline{n}}=x\left(  x-1\right)  \cdots\left(
x-n+1\right), x^{\underline{0}}=1$$
and the rising factorial denoted by $x^{\overline {n}}$, is defined by
 $$x^{\overline{n}}=x\left(  x+1\right)  \cdots\left(
x+n-1\right), x^{\overline{0}}=1. $$ 
\indent The (signed) Stirling numbers of the first kind denoted $s\left(  n,k\right)$
are the coefficients in the expansion
\begin{equation}
x^{\underline{n}}=
{\displaystyle\sum\limits_{k=0}^{n}}
s\left(  n,k\right)  x^{k}.\label{firstkind}
\end{equation}
The exponential generating function is
\begin{align}
 \frac{1}{k!}\left(  \ln\left(
1+t\right)  \right)  ^{k}
&=
{\displaystyle\sum\limits_{n\geq k}}
s\left(  n,k\right) \frac{t^{n}}{n!},\label{genefunfirstkind}
\end{align}
and $s\left(  n,k\right)$ satisfy the following recurrence relation:

\begin{equation}
s\left(  n+1,k\right)=  s\left(  n,k-1\right)-ns\left(  n,k\right)\label{relrecufirstkind}
\end{equation}
and that

\begin{center}
  $s\left(  n,0\right) = \delta_{n,0}$ $(n\in \mathbb{N}
),$  $s\left(  n,k\right)=0 $  $(k > n$ or $k<0),$
\end{center}
where $\delta_{n,m}$  denoted Kronecker symbol.\\
\indent   The Stirling numbers of the second kind denoted $
\genfrac{\{}{\}}{0pt}{}{n}{k} $ count
the number of ways to partition a set of $n$ things into  $k$ nonempty subsets. Explicitly
 $
\genfrac{\{}{\}}{0pt}{}{n}{k}
$ are the coefficients in the expansion
\[
x^{n}=
{\displaystyle\sum\limits_{k=0}^{n}}
\genfrac{\{}{\}}{0pt}{}{n}{k}
x^{\underline{k}}.
\]
\indent The $r$-Stirling numbers \cite{Broder} denotes $
\genfrac{\{}{\}}{0pt}{}{n}{k}_{r}$, for any positive $r \in \mathds{N}$, the number of partitions of a set
of $n$ objects into exactly $k$ nonempty, disjoint subsets, such that the first $r$ elements
are in distinct subsets. These numbers obey the recurrence relation
\begin{equation}
\begin{tabular}
[c]{lll}
$
\genfrac{\{}{\}}{0pt}{0}{n}{k}
_{r}=0,$ &  & $n<r,$\\
$
\genfrac{\{}{\}}{0pt}{0}{n}{k}
_{r}=\delta_{k,r},$ &  & $\text{ }n=r,$\\
$
\genfrac{\{}{\}}{0pt}{0}{n}{k}
_{r}=k
\genfrac{\{}{\}}{0pt}{0}{n-1}{k}
_{r}+
\genfrac{\{}{\}}{0pt}{0}{n-1}{k-1}
_{r},$ &  & $n>r$\label{tab}
\end{tabular}
\end{equation}
and
\begin{equation}
\genfrac{\{}{\}}{0pt}{0}{n}{k}
_{r}=
\genfrac{\{}{\}}{0pt}{0}{n}{k}
_{r-1}-(r-1)
\genfrac{\{}{\}}{0pt}{0}{n-1}{k}_{r-1}.\label{tig}
\end{equation}
The exponential generating function is given by
\begin{equation}
\frac{1}{k!}e^{rt}\left(  e^{t}-1\right)  ^{k}=
{\displaystyle\sum\limits_{n\geq k}}
\genfrac{\{}{\}}{0pt}{0}{n+r}{k+r}%
_{r}\frac{t^{n}}{n!}.\label{rstigenexp}
\end{equation}

\section{The Generalized Fubini transform}
\begin{theorem}
Given an initial sequence
$\left(a_{0,m}\right)_{m\geq0}$, define the matrix
$\mathcal{M}$ associated with the initial \ sequence by (\ref{rec1}) then
\begin{enumerate}
\item The entries of the matrix $\mathcal{M}$ are given by
\begin{equation}
a_{n,m}(x,y)  =\frac{1}{m!}
{\displaystyle\sum\limits_{k=0}^{n}}
\genfrac{\{}{\}}{0pt}{}{n+m}{k+m}
_{m}\left(  k+m\right)!y^{n-k}x^{k}
a_{0,m+k}. \label{explicit1}
\end{equation}
\item Suppose that the initial sequence $a_{0,m+r}$ has the following
ordinary generating function
$
A_{r}\left(  t\right)  ={\sum\limits_{k\geq0}}a_{0,k+r}t^{k}
$.
Then, the sequence $\left(  a_{n,r}\left(  x\right)  \right)  _{n\geq0}$ of the
$r$th columns of the matrix $\mathcal{M}$ has an exponential generating
function
$
B_{r}\left(  t; x, y\right)  ={\sum\limits_{n\geq0}}a_{n,r}(x,y)
\dfrac{t^{n}}{n!},
$
given by
\begin{equation}
B_{r}\left(  t; x, y\right) =\frac{e^{rty}}{r!}\left( e^{-ty}\frac{d}{dt}\right)  ^{r}
\left[
\left(\frac{  e^{ty}-1}{y}
\right)^{r}
A_{r}\left(
\frac{x}{y}(e^{ty}-  1)\right)
\right]
.  \label{anis}
\end{equation}
\end{enumerate}
\end{theorem}

\begin{proof}
\begin{enumerate}
\item  We prove the relation (\ref{explicit1}) by induction on $n$. The result clearly holds for $n = 0$,
 we now show that the formula for $n + 1$ follows from (\ref{rec1}) and induction hypothesis
\begin{align*}
a_{n+1,m}(x,y)
=& \frac{1}{m!}
{\displaystyle\sum\limits_{k=0}^{n-1}}
\genfrac{\{}{\}}{0pt}{}{n+m+1}{k+m+1}
_{m+1}\left(  k+m+1\right)  !x^{k+1}y^{n-k}
a_{0,m+k+1}\\+&
m
\genfrac{\{}{\}}{0pt}{}{n+m}{m}_{m}
y^{n+1}a_{0,m}+ \frac{1}{(m-1)!}
{\displaystyle\sum\limits_{k=1}^{n}}
\genfrac{\{}{\}}{0pt}{}{n+m}{k+m}
_{m}\left(  k+m\right)  !x^{k}y^{n-k+1}
a_{0,m+k}
\\
+& \frac{1}{m!}\genfrac{\{}{\}}{0pt}{}{n+m+1}{n+m+1}
_{m+1}\left(  n+m+1\right)  !x^{n+1}
a_{0,m+n+1}.\\
\end{align*}

After some rearrangements, we get
\begin{align*}
a_{n+1,m}(x,y)
=& \frac{1}{m!}
{\displaystyle\sum\limits_{k=1}^{n}}
\genfrac{\{}{\}}{0pt}{}{n+m+1}{k+m}
_{m+1}\left(  k+m\right)  !x^{k}y^{n-k+1}
a_{0,m+k}+
m
\genfrac{\{}{\}}{0pt}{}{n+m}{m}
_{m}y^{n+1}
a_{0,m}
\\
+& \frac{1}{(m-1)!}
{\displaystyle\sum\limits_{k=1}^{n}}
\genfrac{\{}{\}}{0pt}{}{n+m}{k+m}
_{m}\left(  k+m\right)  !x^{k}y^{n-k+1}a_{0,m+k}
\\
+& \frac{1}{m!}\genfrac{\{}{\}}{0pt}{}{n+m+1}{n+m+1}
_{m+1}\left(  n+m+1\right)  !x^{n+1}a_{0,m+n+1}.
\end{align*}
From (\ref{tab}) and (\ref{tig}),  and after some rearrangements, we get

\begin{align*}
a_{n+1,m}(x,y)
=& \frac{1}{m!}
{\displaystyle\sum\limits_{k=1}^{n}}
\left(
\genfrac{\{}{\}}{0pt}{}{n+m+1}{k+m}_{m+1}+m\genfrac{\{}{\}}{0pt}{}{n+m}{k+m}
_{m}\right)
\left(  k+m\right)  !x^{k}y^{n-k+1}
a_{0,m+k}
\\
+&
\genfrac{\{}{\}}{0pt}{}{n+m+1}{m}
_{m}y^{n+1}
a_{0,m}
+ \frac{1}{m!}\genfrac{\{}{\}}{0pt}{}{n+m+1}{n+m+1}
_{m}\left(  n+m+1\right)  !x^{n+1}a_{0,m+n+1}\\
=&\frac{1}{m!}
{\displaystyle\sum\limits_{k=0}^{n+1}}
\genfrac{\{}{\}}{0pt}{}{n+m+1}{k+m}_{m}
\left(  k+m\right)  !x^{k}y^{n-k+1}
a_{0,m+k}.
\end{align*}
\item The verification of (\ref{anis}) follows by induction on $n$. By using
(\ref{explicit1}), we obtain
\begin{align}
B_{r}\left(  t; x, y\right)   &  ={\sum\limits_{n\geq0}}\left(
\frac{1}{r!}
{\displaystyle\sum\limits_{k=0}^{n}}
\genfrac{\{}{\}}{0pt}{}{n+r}{k+r}
_{r}\left(  k+r\right)  !x^{k}y^{n-k}
a_{0,r+k}
\right)
\dfrac{t^{n}}
{n!}\nonumber\\
 &  ={\sum\limits_{k\geq0}}
\frac{\left(  k+r\right)!}{r!}\left(\frac{x}{y}\right)^{k}
a_{0,r+k}
{\displaystyle\sum\limits_{n\geq k}}
\genfrac{\{}{\}}{0pt}{}{n+r}{k+r}_{r}
\dfrac{\left(ty\right)^{n}}{n!}\nonumber.
\end{align}

From the relation (\ref{rstigenexp}), we obtain
\begin{align}
B_{r}\left(  t; x, y\right)
 &  ={\sum\limits_{k\geq0}}\frac{\left(  k+r\right)!}{r!}\left(\frac{x}{y}\right)^{k}
a_{0,r+k}
\dfrac{1}
{k!}e^{rty}\left(e^{ty}-1\right)^{k}\nonumber
\\
 &=e^{rty}{\sum\limits_{k\geq0}}\dbinom{k+r}{r}a_{0,r+k}
\left(\frac{x}{y}\left(e^{ty}-1\right)\right)^{k}\nonumber.
\end{align}

Since
\[
\binom{k+r}{r}\left[\frac{x}{y}\left(  e^{ty}-1\right)\right]  ^{k}=
\frac{1}{r!x^{r}}\left( e^{-ty}\frac{d}{dt}\right)  ^{r}\left[\frac{x}{y}\left(  e^{ty}-1\right)\right]
^{k+r} ,
\]
we get
\[
B_{r}\left(  t; x, y\right) =\frac{e^{rty}}{r!}\left( e^{-ty}\frac{d}{dt}\right)^{r}
\left[
\left(\frac{  e^{ty}-1}{y}
\right)^{r}
A_{r}\left(
\frac{x}{y}(e^{ty}-  1)\right)
\right].
\]
This evidently completes the proof of Theorem.
\end{enumerate}
\end{proof}
The following corollary represents another expression for the generating function $B_{r}$
\begin{corollary}
\begin{equation}
B_{r}\left(  t; x, y\right)  = \frac{1}{r!}
{\displaystyle\sum\limits_{k=0}^{r}}
s\left(  r,k\right)
 \frac{d^{k}}{dt^{k}}\left[
\left(\frac{  e^{ty}-1}{y}
\right)^{r}
A_{r}\left(
\frac{x}{y}(e^{ty}-  1)\right)
\right].
\label{secondexpressionforB_s}
\end{equation}
\end{corollary}
To prove formula (\ref{secondexpressionforB_s}) using
\[
\left(
e^{-ty}\frac{d}{dt}\right)  ^{r}F\left(  t\right)=
e^{-rty}
{\displaystyle\sum\limits_{k=0}^{r}}
 \genfrac{[}{]}{0pt}{}{r}{k} \frac{d^{k}}{dt^{k}}F\left(  t\right),
\]
with
\[
F\left(  t\right)=
\left(\frac{  e^{ty}-1}{y}
\right)^{r}
A_{r}\left(
\frac{x}{y}(e^{ty}-  1)\right).
\]
\begin{theorem}
Given final sequence $\left(a_{n,0}\right)_{n\geq0}$, define the matrix $\mathcal{M}$
associated with the final sequence by
\begin{equation}
a_{n,m+1}(x, y)  =\frac{1}{x\left(m+1\right)}\left(  a_{n+1,m}(x, y)
-ym  a_{n,m}(x, y) \right),  \label{rec2}
\end{equation}
 then
\begin{enumerate}
\item The entries of the matrix $\mathcal{M}$ are given by
\begin{equation}
a_{n,m} \left(x, y\right) =\frac{y^{m}}{x^{m}m!}
{\displaystyle\sum\limits_{k=0}^{m}}
y^{-k}
s\left(  m,k\right)
a_{n+k,0\text{
\ }}. \label{explicit2}
\end{equation}
\item Suppose that the final sequence $a_{n+r,0}$ has the following
exponential generating function $
\mathcal{\widehat{B}}_{r}\left(  t\right)  ={\sum\limits_{k\geq0}}a_{k+r,0}\frac
{t^{k}}{k!}$.
Then, the sequence $\left(  a_{r,m}\left(  x\right)  \right)  _{m\geq0}$ of the
$r$th row of the matrix $\mathcal{M}$ has an ordinary generating function
$
\mathcal{\widehat{A}}_{r}\left(  t; x, y\right)  ={\sum\limits_{m\geq0}}a_{r,m}(x, y)  t^{m},
$ given by
\begin{equation}
\mathcal{\widehat{A}}_{r}\left( t; x, y\right)  =
\mathcal{\widehat{B}}_{r}\left( y^{-1}\ln\left(1+ \frac{ty}{x}\right)\right).
\end{equation}
\end{enumerate}
\end{theorem}
\begin{proof}
\begin{enumerate}
\item
We prove by induction on $m$, the result clearly holds for $m  =0$. By induction hypothesis and  (\ref{rec2}), we have

\begin{align*}
a_{n,m+1}(x, y)
&  =\frac{y^{m}}{x^{m+1}\left(m+1\right)!}
\left(
{\displaystyle\sum\limits_{k=0}^{m}}y^{-k}
s\left(  m,k\right)
a_{n+k+1,0\text{
\ }}-ym
{\displaystyle\sum\limits_{k=0}^{m}}y^{-k}
s\left(  m,k\right)
a_{n+k,0\text{
\ }}\right)  \\
&  =\frac{y^{m}}{x^{m+1}\left(m+1\right)!}
\left(
y^{-m}s\left(  m,m\right)
a_{n+m+1,0\text{
\ }}+
{\displaystyle\sum\limits_{k=0}^{m-1}}
y^{-k}s\left(  m,k\right)
a_{n+k+1,0\text{
\ }} \right) \\
&-\frac{y^{m}}{x^{m+1}\left(m+1\right)!}
\left( m {\displaystyle\sum\limits_{k=1}^{m}}
y^{-k+1}s\left(  m,k\right)
a_{n+k,0\text{
\ }}+my^{}s\left(  m,0\right)
a_{n,0\text{
\ }}\right).
\end{align*}

After some rearrangements, we get

\begin{align*}
a_{n,m+1}(x, y)
&  =\frac{y^{m}}{x^{m+1}\left(m+1\right)!}
\left(
{\displaystyle\sum\limits_{k=1}^{m}}
y^{-k+1}s\left(  m,k-1\right)
a_{n+k,0\text{
\ }}-m {\displaystyle\sum\limits_{k=1}^{m}}
y^{-k+1}s\left(  m,k\right)
a_{n+k,0\text{
\ }} \right) \\
&+\frac{y^{m}}{x^{m+1}\left(m+1\right)!}
\left(
y^{-m}s\left(  m,m\right)
a_{n+m+1,0\text{
\ }}-my^{}s\left(  m,0\right)
a_{n,0\text{
\ }}\right).
\end{align*}

From (\ref{relrecufirstkind}) and after some rearrangements, we get
\begin{align*}
a_{n,m+1}(x, y)
&  =\frac{y^{m+1}}{x^{m+1}\left(m+1\right)!}
{\displaystyle\sum\limits_{k=1}^{m}}
y^{-k}
\left( s\left(  m,k-1\right)
-ms\left(  m,k\right)
\right)
a_{n+k,0\text{
\ }} \\
&+\frac{y^{m+1}}{x^{m+1}\left(m+1\right)!}
\left(
y^{-m-1}s\left(  m+1,m+1\right)
a_{n+m+1,0\text{
\ }}+s\left(  m+1,0\right)
a_{n,0\text{
\ }} \right)\\
&  =\frac{y^{m+1}}{x^{m+1}\left(m+1\right)!}
{\displaystyle\sum\limits_{k=0}^{m+1}}
y^{-k}s\left(  m+1,k\right)
a_{n+k,0\text{\ }}.
\end{align*}
which completes the proof.
\item According to (\ref{explicit2}), we have
\begin{align*}
\mathcal{\widehat{A}}_{r}\left( t; x, y\right)   &  ={\sum\limits_{m\geq0}}\left(  \frac{y^{m}}{x^{m}m!}
{\displaystyle\sum\limits_{k=0}^{m}}
y^{-k}s\left(  m,k\right)
a_{r+k,0\text{
\ }}\right)  t^{m}\\
&  ={\sum\limits_{k\geq0}}a_{r+k,0\text{
\ }}y^{-k}{\sum\limits_{m\geq k}}\frac{y^{m}}{x^{m}m!}
s\left(  m,k\right)
 t^{m}.
\end{align*}

From the relation (\ref{genefunfirstkind}), we obtain

\begin{align*}
\mathcal{\widehat{A}}_{r}\left( t; x, y\right)
&  ={\sum\limits_{k\geq0}}a_{r+k,0\text{
\ }}y^{-k}
\frac{1}{k!}\left(  \ln\left(
1+\frac{ty}{x}\right)  \right) ^{k}\\
&= \mathcal{\widehat{B}}_{r}\left( y^{-1}\ln\left(1+ \frac{ty}{x}\right)\right),
\end{align*}

which completes the proof.
\end{enumerate}
\end{proof}
\begin{corollary}
\label{Co1}For $n,m\geq0$, we have%
\begin{equation}
{\displaystyle\sum\limits_{k=0}^{n}}
\genfrac{\{}{\}}{0pt}{}{n+m}{k+m}
_{m}\left(  k+m\right)  !x^{k}y^{n-k}
a_{0,m+k}=
{\displaystyle\sum\limits_{k=0}^{m}}
s\left(  m,k\right)
x^{-m}y^{m-k}a_{n+k,0\text{
\ }} .\label{cora}
\end{equation}
\end{corollary}
The identity (\ref{cora}) can be viewed as the generalized Fubini transform which can be reduced, for $m=0$  to the Fubini transform of the sequence $\alpha_{n}
$, and for $n=0$  to the inverse Fubini transform of the sequence $\beta_{m}$.

\section{On generalized Fubini polynomials}
Setting the initial sequence $a_{0,m}=1$ in (\ref{rec1}), we get the
following matrix
\small{
\[
\mathcal{M}=
\begin{pmatrix}
1 & 1 & 1 & 1 & \cdots\\
x & y+2x & 2y+3x & 3y+4x & \cdots\\
2x^{2}+yx & 6x^{2}+6yx+y^{2} & 12x^{2}+15xy+4y^{2} & \vdots & \\
6x^{3}+6x^{2}y+xy^{2} & 24x^{3}+36x^{2}y+14xy^{2}+y^{3} & \vdots &  & \\
24x^{4}+36x^{3}y+14x^{2}y^{2}+xy^{3} & \vdots &  &  & \\
\vdots &  &  &  &
\end{pmatrix}.
\]
}
\normalsize
Since $A_{r}\left(  t\right)  =\frac{1}{1-t},$ it follows from (\ref{anis}) and (\ref{secondexpressionforB_s}) that the
final sequence has an exponential generating function given by
\begin{align*}
B_{r}\left( t,x,y\right)  &=\frac{e^{rty}}{r!} \left( e^{-ty}\frac{d}{dt}\right)  ^{r} \left[
\left(\frac{  e^{ty}-1}{y}\right) ^{r}
\frac{1}{1-\frac{x}{y}\left(  e^{ty}-1\right)}\right]  \\
=&\frac{1}{r!}
{\displaystyle\sum\limits_{k=0}^{r}}
s\left(  r,k\right)
 \frac{d^{k}}{dt^{k}}\left[
\left(\frac{  e^{ty}-1}{y}\right) ^{r}
\frac{1}{1-\frac{x}{y}\left(  e^{ty}-1\right)}\right].
\end{align*}
In particular for $r=0$, we have
\begin{align*}
B_{0}\left(  t;x,y\right)   &  = \frac{1}{1-\frac{x}{y}\left(  e^{ty}-1\right)}.
\end{align*}
\begin{definition}
We defined a sequence of polynomials $\mathfrak{F}_{n} (x, y)$ of two variables $x, y$, called generalized Fubini polynomials, by means of
the generating function
\begin{align}
\frac{1}{1-\frac{x}{y}(e^{ty}-1)}&=
{\displaystyle\sum\limits_{n\geq 0}}
\mathfrak{F}_{n} \left(  x, y\right)  \frac{t^{n}}{n!}.\label{genefunG(x,y)2var}
\end{align}
\end{definition}
The explicit formula for $\mathfrak{F}_{n} \left(  x, y\right)$ is given by
\begin{align}
\mathfrak{F}_{n} \left(  x, y\right)   = {\displaystyle\sum\limits_{k=0}^{n}}\genfrac{\{}{\}}{0pt}{}{n}{k}
k!x^{k}y^{n-k} .\label{explicitFubini}
\end{align}
By setting $y=1$ in \eqref{genefunG(x,y)2var}, we get
\begin{align}
\frac{1}{1-x\left(e^{t}-  1\right)} &=
{\displaystyle\sum\limits_{n\geq 0}}
\mathfrak{F}_{n} \left(  x, 1\right)  \frac{t^{n}}{n!}  \nonumber\\
&  ={\sum\limits_{n\geq0}}\omega_{n}\left(  x\right)  \dfrac{t^{n}}{n!},\label{genefunFubbini}
\end{align}
where $\omega_{n}\left(  x\right)$ denotes the  Fubini polynomials \cite{boyadzhiev2005series,kargin2017some,qi2019determinantal}, defined by
\[
\omega_{n}\left(  x\right)  =
{\displaystyle\sum_{k=0}^{n}}
\genfrac{\{}{\}}{0pt}{}{n}{k}k!
x^{k}.
\]
By (\ref{genefunG(x,y)2var}) and (\ref{genefunFubbini}), we can write the relation between $\omega_{n}\left(  x\right)$ and $\mathfrak{F}_{n} \left(  x, y\right)$, given by the following two formulas

\begin{align}
\mathfrak{F}_{n} \left(  x, y\right)=y^{n}\omega_{n}\left(  \frac{x}{y}\right)
\end{align}
and
\begin{align}
\omega_{n}\left(  x\right)=y^{-n}\mathfrak{F}_{n} \left(  xy, y\right).
\end{align}
\indent The Fubini polynomials $\omega_{n}\left(  x\right)$ are related to the geometric series in the following way \cite{boyadzhiev2005series,boyadzhiev2016geometric}
\begin{equation}
\left(x \frac{d}{dx}\right)^{n} \frac{1}{1-x}
={\displaystyle\sum\limits_{k\geq0}}
x ^{k}    k^{n}
=
\frac{1}{1-x}
\omega_{n} \left( \frac{x}{1-x}\right).\label{operatorboyadzhiev}
\end{equation}
This relation can be extended to a more general form depending on two variables $x$ and $y$.
\begin{theorem}
 For $x$ different to $y$, the polynomials $\mathfrak{F}_{n} \left(  x, y\right)$ have the following property
\begin{equation}
\frac{y}{y-x}
\mathfrak{F}_{n} \left(  \frac{ xy}{y-x}, y\right)
={\displaystyle\sum\limits_{k\geq0}}
\left(\frac{x}{y} \right)^{k} \left(   yk\right)^{n}
=y^{n}\left(x \frac{d}{dx}\right)^{n} \frac{y}{y-x}.\label{operatorgeneralized}
\end{equation}
\end{theorem}
\begin{proof}
We have
\begin{align}
\frac{y}{y-x}
{\displaystyle\sum\limits_{n\geq 0}}
\mathfrak{F}_{n} \left( \frac{ xy}{y-x}, y\right)  \frac{t^{n}}{n!} \nonumber
&  =\frac{y}{y-x}\left(\frac{1}{1- \frac{x}{y-x}      ( e^{ty}-1)}\right)\nonumber\\
&  =\frac{1}{1- \frac{x}{y}       e^{ty}}\nonumber\\
&  ={\displaystyle\sum\limits_{k\geq 0}}
\left(\frac{x}{y} \right)^{k}
 \left(e^{ty}\right)^{k}\nonumber.
\end{align}

Then

\begin{align}
\frac{y}{y-x}
{\displaystyle\sum\limits_{n\geq 0}}
\mathfrak{F}_{n} \left( \frac{ xy}{y-x}, y\right)  \frac{t^{n}}{n!} \nonumber
&  ={\displaystyle\sum\limits_{k\geq 0}}
\left(\frac{x}{y} \right)^{k}
 {\displaystyle\sum\limits_{n\geq 0}} (ky)^{n}\frac{ t^{n}}{n!}
 \nonumber\\
&  ={\displaystyle\sum\limits_{n\geq 0}}
\left({\displaystyle\sum\limits_{k\geq 0}}
\left(\frac{x}{y} \right)^{k}
  (ky)^{n}\right)\frac{ t^{n}}{n!}.\nonumber
\end{align}

Equating the coefficients of $\frac{ t^{n}}{n!}$, we get
\begin{align}
\frac{y}{y-x}
\mathfrak{F}_{n} \left( \frac{ xy}{y-x}, y\right)  
&  ={\displaystyle\sum\limits_{k\geq 0}}
\left(\frac{x}{y} \right)^{k}
  (ky)^{n}.\nonumber
\end{align}

On the other hand, we apply the formula (4.1) in \cite{boyadzhiev2005series}, we get
\begin{align}
{\displaystyle\sum\limits_{k\geq0}}
\left(\frac{x}{y} \right)^{k} \left(   yk\right)^{n}
&=y^{n}{\displaystyle\sum\limits_{k\geq0}}\left(x \frac{d}{dx}\right)^{n} \left(\frac{x}{y} \right)^{k}\nonumber\\
&=y^{n}\left(x \frac{d}{dx}\right)^{n}{\displaystyle\sum\limits_{k\geq0}} \left(\frac{x}{y} \right)^{k}\nonumber\\
&=y^{n}\left(x \frac{d}{dx}\right)^{n} \frac{y}{y-x}.
\end{align}
This evidently completes the proof of the theorem.
\end{proof}
\begin{remark}
By setting $y = 1$ in \eqref{operatorgeneralized} we get \eqref{operatorboyadzhiev}.
\end{remark}
\indent Now,  recall that the exponential generating function for Bell polynomials $\phi_{n}(x)$, is given by
\begin{align}%
 e^{x(e^{t}-1)} &  =
{\displaystyle\sum\limits_{n\geq 0}}
\phi_{n}(x)  \frac{t^{n}}{n!}\label{genefunBell}
\end{align}
and given explicitly by
\begin{align}
\phi_{n}(x)    &  = {\displaystyle\sum\limits_{k=0}^{n}}
\genfrac{\{}{\}}{0pt}{}{n}{k}
x^{k}.\label{explicitBell}
\end{align}
In the following result, we will give the integral representation for $\mathfrak{F}_{n} \left(  x, y\right)$ and the link with $\phi_{n}(x)$.
\begin{theorem}
For $n\geq0$, we have
\begin{align}%
\mathfrak{F}_{n} \left(  x, y\right)
 &=y^{n}\int_{0}^{+\infty}\phi\left(\frac{x}{y}\lambda\right)e^{-\lambda}d\lambda \label{FubinigeneandBell}%
 \end{align}
and
\begin{align}%
{\displaystyle\sum\limits_{n\geq 0}}
\mathfrak{F}_{n} \left(  x, y\right)  \frac{t^{n}}{n!}
 &=\int_{0}^{+\infty}e^{-\lambda\left(1-\frac{x}{y}(e^{ty}-1)\right) }d\lambda.\label{Fubinigeneandintegral}%
\end{align}
\end{theorem}
\begin{proof}
Replacing $x$ by $\frac{x}{y}\lambda$ in (\ref{explicitBell}) and multiplying both sides by $y^{n}e^{-\lambda}$ and integrating for $\lambda$ from zero to infinity, we have
\begin{align}
 y^{n}\int_{0}^{+\infty}\phi\left(\frac{x}{y}\lambda\right)e^{-\lambda}d\lambda
\nonumber
&   =y^{n}\int_{0}^{+\infty}\left(
 {\displaystyle\sum\limits_{k=0}^{n}}
\genfrac{\{}{\}}{0pt}{}{n}{k}
 \left(
 \frac{x}{y}\lambda\right)^{k} \right)e^{-\lambda}  d\lambda
\nonumber\\
&   = y^{n}
 {\displaystyle\sum\limits_{k=0}^{n}}
\genfrac{\{}{\}}{0pt}{}{n}{k}
 \left(
 \frac{x}{y}\right)^{k}
 \int_{0}^{+\infty}e^{-\lambda}\left(\lambda\right)^{k}  d\lambda
\nonumber\\
&   = y^{n}
 {\displaystyle\sum\limits_{k=0}^{n}}
\genfrac{\{}{\}}{0pt}{}{n}{k}
 \left(
 \frac{x}{y}\right)^{k} k!\nonumber.
\end{align}
By comparing with (\ref{explicitFubini}) we get (\ref{FubinigeneandBell}).

Now to prove (\ref{Fubinigeneandintegral}), using (\ref{FubinigeneandBell}), we have
\begin{align}
{\displaystyle\sum\limits_{n\geq 0}}
\mathfrak{F}_{n} \left(  x, y\right)  \frac{t^{n}}{n!}  &  ={\displaystyle\sum\limits_{n\geq 0}}
\left( y^{n}\int_{0}^{+\infty}\phi\left(\frac{x}{y}\lambda\right)e^{-\lambda}d\lambda \right)  \frac{t^{n}}{n!}
\nonumber\\
&   =\int_{0}^{+\infty}\left( e^{-\lambda} {\displaystyle\sum\limits_{n\geq 0}}
 \phi\left(\frac{x}{y}\lambda\right)\frac{(ty)^{n}}{n!} \right) d\lambda
\nonumber
\end{align}
we apply (\ref{genefunBell}), we get
\begin{align}
{\displaystyle\sum\limits_{n\geq 0}}
\mathfrak{F}_{n} \left(  x, y\right)  \frac{t^{n}}{n!}
&  =\int_{0}^{+\infty}\left( e^{-\lambda}
  e^{\frac{x}{y}\lambda(e^{ty}-1)} \right) d\lambda\nonumber\\
  &= \int_{0}^{+\infty}e^{-\lambda\left(1-\frac{x}{y}(e^{ty}-1)\right) }d\lambda.
\nonumber
\end{align}

\end{proof}
\begin{remark}By setting $y=1$ in  (\ref{FubinigeneandBell}) and (\ref{Fubinigeneandintegral}), respectively, we get (3.11) and (3.13) in \cite{boyadzhiev2005series}.
\end{remark}
The Fubini polynomials of two variables $\omega_{n}\left(  x, y\right)$ are defined in \cite{kargin2017some,kargin2018p,qi2019determinantal} by the following generating function
\begin{align}
 \frac{e^{ty}}{1-x(e^{t}-1)}&=
{\displaystyle\sum\limits_{n\geq 0}}
\omega_{n}\left(  x, y\right)  \frac{t^{n}}{n!}.\label{genefunfubini2var}
\end{align}
The next result represents the relation between $\mathfrak{F}_{n} \left(  x, y\right)$ and $\omega_{n}\left(  x, y\right)$.
\begin{theorem} \label{theorem8}
For $n\geq0$, we have
\begin{align}
\mathfrak{F}_{n} \left(  x, y\right)=
y^{n}
{\displaystyle\sum\limits_{k= 0}^{n}}y^{k}
\genfrac{(}{)}{0pt}{}{n}{k}%
(-1)^{k}
\omega_{n-k} \left(  \frac{x}{y}, y\right).
\end{align}
\end{theorem}
\begin{proof} From (\ref{genefunG(x,y)2var}) and (\ref{genefunfubini2var}), we have

\begin{align}
{\displaystyle\sum\limits_{n\geq 0}}
\mathfrak{F}_{n} \left(  x, y\right)  \frac{t^{n}}{n!}
&  =e^{-ty^{2}}\frac{e^{ty^{2}}}{1-\frac{x}{y}(e^{ty}-1)}
\nonumber\\
&  =\left(
{\displaystyle\sum\limits_{n\geq 0}}
  \frac{(-ty^{2})^{n}}{n!}\right)
  \left(
{\displaystyle\sum\limits_{n\geq 0}}
\omega_{n} \left(  \frac{x}{y}, y\right)  \frac{(ty)^{n}}{n!}\right)
\nonumber\\
&  = {\displaystyle\sum\limits_{n\geq 0}}
\left( {\displaystyle\sum\limits_{k= 0}^{n}}
\genfrac{(}{)}{0pt}{}{n}{k}
(-y^{2})^{k}
\omega_{n-k} \left(  \frac{x}{y}, y\right) y^{n-k} \right)\frac{t^{n}}{n!}
\nonumber\\
&  = {\displaystyle\sum\limits_{n\geq 0}}
\left( y^{n}
{\displaystyle\sum\limits_{k= 0}^{n}}
\genfrac{(}{)}{0pt}{}{n}{k}
(-1)^{k}y^{k}
\omega_{n-k} \left(  \frac{x}{y}, y\right)
\right)
 \frac{t^{n}}{n!}, \nonumber%
\end{align}
that is to say
\[
\mathfrak{F}_{n} \left(  x, y\right)=
y^{n}
{\displaystyle\sum\limits_{k= 0}^{n}}y^{k}
\genfrac{(}{)}{0pt}{}{n}{k}
(-1)^{k}
\omega_{n-k} \left(  \frac{x}{y}, y\right).
\]
\end{proof}
In addition to the above properties of $\mathfrak{F}_{n} \left(  x, y\right)$ polynomials, we now present some recurrence relations. The following lemma will be useful for the proof of the next theorem.

\begin{lemma} For nonzero complex numbers x and y, we have
\begin{align}
\frac{e^{ty}}{1-\frac{x}{y}\left(  e^{ty}-1\right)}   & = \left( \frac{1}{x}-\frac{1}{y}(e^{ty}-1)\right)\frac{d}{dt}\left(
\frac{1}{1-\frac{x}{y}\left(  e^{ty}-1\right)}\right)\label{lemm1rec1}.
\end{align}
\end{lemma}
\begin{theorem}
For $n\geq0$, we have
\begin{equation}\label{recurrencefubinigeneralized1}
\mathfrak{F}_{n+1} \left(  x, y\right)= \left(  \frac{xy}{x+y}\right)
{\displaystyle\sum\limits_{k=0}^{n}}\genfrac{(}{)}{0pt}{}{n}{k}y^{n-k}\left(
\mathfrak{F}_{k} \left(  x, y\right) +\frac{1}{y}
\mathfrak{F}_{k+1} \left(  x, y\right) \right).
\end{equation}
\end{theorem}
\begin{proof}
Using the above lemma, then (\ref{lemm1rec1}) is equivalent to
\begin{align*}
{\displaystyle\sum\limits_{n\geq 0}}\left(
{\displaystyle\sum\limits_{k=0}^{n}}\genfrac{(}{)}{0pt}{}{n}{k}y^{n-k}
\mathfrak{F}_{k} \left(  x, y\right) \right)\frac{t^{n}}{n!}   & = \left( \frac{1}{x}-\frac{1}{y}{\displaystyle\sum\limits_{n\geq 0}} \frac{(ty)^{n}}{n!}+\frac{1}{y}\right){\displaystyle\sum\limits_{n\geq 0}}
\mathfrak{F}_{n+1} \left(  x, y\right)  \frac{t^{n}}{n!}.
\end{align*}
Then,
\begin{align*}
{\displaystyle\sum\limits_{n\geq 0}}\left(
{\displaystyle\sum\limits_{k=0}^{n}}\genfrac{(}{)}{0pt}{}{n}{k}y^{n-k}
\mathfrak{F}_{k} \left(  x, y\right) \right)\frac{t^{n}}{n!}
& =
 \frac{1}{x}{\displaystyle\sum\limits_{n\geq 0}}
\mathfrak{F}_{n+1} \left(  x, y\right)  \frac{t^{n}}{n!}-\frac{1}{y}
{\displaystyle\sum\limits_{n\geq 0}} \frac{(ty)^{n}}{n!}
{\displaystyle\sum\limits_{n\geq 0}}
\mathfrak{F}_{n+1} \left(  x, y\right)  \frac{t^{n}}{n!}
\\
&+\frac{1}{y}
{\displaystyle\sum\limits_{n\geq 0}}
\mathfrak{F}_{n+1} \left(  x, y\right)  \frac{t^{n}}{n!}\\
& =
 \frac{1}{x}{\displaystyle\sum\limits_{n\geq 0}}
\mathfrak{F}_{n+1} \left(  x, y\right)  \frac{t^{n}}{n!}-\frac{1}{y}
{\displaystyle\sum\limits_{n\geq 0}}\left(
{\displaystyle\sum\limits_{k=0}^{n}}\genfrac{(}{)}{0pt}{}{n}{k}y^{n-k}
\mathfrak{F}_{k+1} \left(  x, y\right)  \right)\frac{t^{n}}{n!}\\
&+\frac{1}{y}
{\displaystyle\sum\limits_{n\geq 0}}
\mathfrak{F}_{n+1} \left(  x, y\right)  \frac{t^{n}}{n!}\\
& = {\displaystyle\sum\limits_{n\geq 0}}\left(
 \frac{1}{x}
\mathfrak{F}_{n+1} \left(  x, y\right) -\frac{1}{y}
{\displaystyle\sum\limits_{k=0}^{n}}\genfrac{(}{)}{0pt}{}{n}{k}y^{n-k}
\mathfrak{F}_{k+1} \left(  x, y\right)
+\frac{1}{y}
\mathfrak{F}_{n+1} \left(  x, y\right)
\right)\frac{t^{n}}{n!}.
\end{align*}
Equating the coefficients of $\frac{t^{n}}{n!}$, we get
\begin{align*}
{\displaystyle\sum\limits_{k=0}^{n}}\genfrac{(}{)}{0pt}{}{n}{k}y^{n-k}
\mathfrak{F}_{k} \left(  x, y\right)  &=
 \frac{1}{x}
\mathfrak{F}_{n+1} \left(  x, y\right) -\frac{1}{y}
{\displaystyle\sum\limits_{k=0}^{n}}\genfrac{(}{)}{0pt}{}{n}{k}y^{n-k}
\mathfrak{F}_{k+1} \left(  x, y\right)
+\frac{1}{y}
\mathfrak{F}_{n+1} \left(  x, y\right)
\end{align*}
and after some rearrangements, we obtain the result.
\end{proof}
\begin{remark}
 As a special case, we get the formula (24) in \cite{dil2011investigating} by setting $y = 1$ in (\ref{recurrencefubinigeneralized1}).
\end{remark}
\begin{theorem}
For $n\geq0$, we have
\begin{equation}\label{recurrencefubinigeneralized}
\mathfrak{F}_{n+1} \left(  x, y\right)+y\mathfrak{F}_{n} \left(  x, y\right)=(x+y)
{\displaystyle\sum\limits_{k=0}^{n}}\genfrac{(}{)}{0pt}{}{n}{k}
\mathfrak{F}_{k} \left(  x, y\right)
\mathfrak{F}_{n-k} \left(  x, y\right).\end{equation}
\end{theorem}
\begin{proof}
Considering the derivative of the generating function of the
polynomials $\mathfrak{F}_{n} \left(  x, y\right)$ (\ref{genefunG(x,y)2var}), we have
\begin{align}
{\displaystyle\sum\limits_{n\geq 0}}
\mathfrak{F}_{n+1} \left(  x, y\right)  \frac{t^{n}}{n!}  &  =\frac{xe^{ty}}{\left(
1-\frac{x}{y}(e^{ty}-1)\right)^{2}
}\nonumber\\
&  =\left(\frac{x+y}{
1-\frac{x}{y}(e^{ty}-1)}
-y\right)\frac{1}{
1-\frac{x}{y}(e^{ty}-1)} \nonumber \\%
&=(x+y){\displaystyle\sum\limits_{n\geq 0}}
\mathfrak{F}_{n} \left(  x, y\right)  \frac{t^{n}}{n!}{\displaystyle\sum\limits_{n\geq 0}}
\mathfrak{F}_{n} \left(  x, y\right)  \frac{t^{n}}{n!}-y{\displaystyle\sum\limits_{n\geq 0}}
\mathfrak{F}_{n} \left(  x, y\right)  \frac{t^{n}}{n!}
\nonumber\\
&=(x+y)
{\displaystyle\sum\limits_{n\geq 0}}
\left( {\displaystyle\sum\limits_{k= 0}^{n}}\genfrac{(}{)}{0pt}{}{n}{k}
\mathfrak{F}_{k} \left(  x, y\right)
\mathfrak{F}_{n-k} \left(  x, y\right)  -y
\mathfrak{F}_{n} \left(  x, y\right)  \right)  \frac{t^{n}}{n!}.\nonumber
\end{align}
Equating the coefficients of $\frac{t^{n}}{n!}$,
and after some rearrangements, we obtain the result.
\end{proof}
For $y=1$, we get the result of the Theorem 1  in \cite{kargin2017some}.
\begin{theorem}
For $n\geq0$ and  for $x_{1}$ different to $ x_{2}$, we have
\begin{equation}\label{2recurrencefubinigeneralized}
{\displaystyle\sum\limits_{k=0}^{n}}\genfrac{(}{)}{0pt}{}{n}{k}
\mathfrak{F}_{k} \left(  x_{1}, y\right)
\mathfrak{F}_{n-k} \left(  x_{2}, y\right)
=\frac{x_{2}\mathfrak{F}_{n} \left(  x_{2}, y\right)-x_{1}\mathfrak{F}_{n} \left(  x_{1}, y\right)}{x_{2}-x_{1}}.
\end{equation}
\end{theorem}
\begin{proof} The proof of (\ref{2recurrencefubinigeneralized}) becomes as follows
\begin{align}
\frac{1}{
1-\frac{x_{1}}{y}(e^{ty}-1)
} \frac{1}{
1-\frac{x_{2}}{y}(e^{ty}-1)
}&  =\frac{x_{2}}{x_{2}-x_{1}}\frac{1}{
1-\frac{x_{2}}{y}(e^{ty}-1)
} \nonumber\\
 &- \frac{x_{1}}{x_{2}-x_{1}}\frac{1}{
1-\frac{x_{1}}{y}(e^{ty}-1)
}\nonumber
.\end{align}
\end{proof}
Now, in this part of the paper, we will connect the polynomials $\mathfrak{F}_{n} \left(  x, y\right)$ with Eulerian polynomials and Frobenius-Euler polynomials.
It is known that for $x\neq1$ and $n \geq 0$, the Eulerian polynomials  $A_{ n}(x)$ and the Frobenius-Euler polynomials $H_{n}(x;y)$  are defined
respectively by the following generating functions \cite{SrivastavaBouticheRahmani, tomaz2014matrix}
\begin{equation}%
\frac{1-x}{
e^{t(x-1)}-x} =
{\displaystyle\sum\limits_{n\geq 0}}
A_{n} \left(  x \right)  \frac{t^{n}}{n!},
\label{genefunEulerianpolynomials1}
\end{equation}

\begin{equation}
\frac{1-x}{
e^{
t}-x}e^{
ty}  =
{\displaystyle\sum\limits_{n\geq 0}}
H_{n} \left(  x;y \right)  \frac{t^{n}}{n!}.\label{genefunFrobeniusEulerianpolynomials}
\end{equation}

\begin{theorem}
For $n\geq 0$,  and for nonzero complex numbers $x$ and $y$,we have
\begin{align}
\mathfrak{F}_{n} \left(  x, y\right)
 & = x^{n}
A_{n} \left(  1+\frac{y}{x}\right)\label{(4.23)}
\end{align}
and for $t\neq 1$, we have
\begin{align}
A_{n} \left(  t\right)
 & = \left(  \frac{t-1}{y}\right)^{n}
\mathfrak{F}_{n} \left(  \frac{y}{t-1}, y\right)
  = \left(  \frac{1}{x}\right)^{n}
\mathfrak{F}_{n} \left(
x, x(t-1)
\right).\label{(4.25)}
\end{align}
\end{theorem}

\begin{proof}
The generating functions (\ref{genefunG(x,y)2var}) and (\ref{genefunEulerianpolynomials1}) can be rewritten as
\begin{align}
\frac{1}{e^{ty}-(1+\frac{y}{x})} &  =
-\frac{x}{y}
{\displaystyle\sum\limits_{n\geq 0}}
\mathfrak{F}_{n} \left(  x, y\right)  \frac{t^{n}}{n!}
\end{align}
and for $x\neq1$
\begin{align}
\frac{1}{
e^{t}-x} &  =
-{\displaystyle\sum\limits_{n\geq 0}}
\frac{A_{n} \left(  x\right)}{(x-1)^{n+1}}  \frac{t^{n}}{n!}
.\label{genefunEulerianpolynomials}
\end{align}
Then,
\begin{align}
\frac{x}{y}\mathfrak{F}_{n} \left(  x, y\right)
 & = y^{n}
\frac{A_{n} \left(  1+\frac{y}{x}\right)}{(\frac{y}{x})^{n+1}}\nonumber.\\
 & = x^{n}
\frac{A_{n} \left(  1+\frac{y}{x}\right)}{(\frac{y}{x})}\nonumber.
\end{align}
Which is equivalent to (\ref{(4.23)}).\\

Now, for $t=1+\frac{y}{x}$ in (\ref{(4.23)}), we obtain (\ref{(4.25)}).

\end{proof}

\begin{theorem}
For $n\geq 0$, we have
\[
\mathfrak{F}_{n} \left(  x, y\right)=
y^{n}
{\displaystyle\sum\limits_{k= 0}^{n}}y^{k}
\genfrac{(}{)}{0pt}{}{n}{k}
(-1)^{k}
H_{n-k} \left( 1+ \frac{y}{x}; y\right).
\]
\end{theorem}
\begin{proof} From the generating functions (\ref{genefunG(x,y)2var}) and (\ref{genefunFrobeniusEulerianpolynomials}), we have,
\begin{align}
{\displaystyle\sum\limits_{n\geq 0}}
\mathfrak{F}_{n} \left(  x, y\right)  \frac{t^{n}}{n!}  &  =e^{-ty^{2}}\frac{\left(1-(1+\frac{y}{x})\right)}{e^{ty}-(1+\frac{y}{x})}e^{ty^{2}}
\nonumber\\
&  =e^{-ty^{2}}{\displaystyle\sum\limits_{n\geq 0}}
H_{n} \left(1+  \frac{x}{y}; y\right)  \frac{(ty)^{n}}{n!}.
\nonumber
\end{align}
In the same way as the proof of Theorem \ref{theorem8}. we get the result.
\end{proof}

\section{PROBABILISTIC REPRESENTATION}
We consider a geometric distributed random variable $X$. The probability density function, for $k\in \mathds{N^{*}}$ and two parameters $p$ and $q$, such that $q=1-p$, as follows:
\[P(X=k)=pq^{k-1}.\]
The higher moment of $X$ is given by
\begin{align} \label{highermomonts}
E(X^{n})=&{\displaystyle\sum\limits_{k\geq 1}}
k^{n}p(1-p)^{k-1}.
\end{align}
In the next paragraph, we show that $\mathfrak{F}_{n} \left(  x, y\right)$  can be viewed as the $n$th moment of a random
variable $X-1$ where $X$ follows the geometric law.
\begin{theorem} Let $X$ be a random variable follows the geometric law and
for $p=\frac{y}{x+y}>0$ , we have
\begin{align}
\mathfrak{F}_{n} \left(  x, y\right)
=&\frac{y}{x+y}{\displaystyle\sum\limits_{k\geq0}}
\left(\frac{x}{x+y} \right)^{k} \left(   yk\right)^{n}\\
=&y^{n}E((X-1)^{n}).
\end{align}
\end{theorem}
\begin{proof}
From (\ref{genefunG(x,y)2var}), we have
\begin{align}
{\displaystyle\sum\limits_{n\geq 0}}
\mathfrak{F}_{n} \left(  x, y\right)  \frac{t^{n}}{n!}
&  =\frac{y}{x+y}\frac{1}{\left(1- \frac{x}{x+y} e^{ty}\right)}\nonumber\\
&  =\frac{y}{x+y}{\displaystyle\sum\limits_{k\geq 0}}
\left(\frac{x}{x+y} \right)^{k}
 \left(e^{ty}\right)^{k}
\nonumber\\
&  =\frac{y}{x+y}{\displaystyle\sum\limits_{k\geq 0}}
\left(\frac{x}{x+y} \right)^{k}
 {\displaystyle\sum\limits_{n\geq 0}} (ky)^{n}\frac{ t^{n}}{n!}
 \nonumber\\
&  =\frac{y}{x+y}{\displaystyle\sum\limits_{n\geq 0}}
\left({\displaystyle\sum\limits_{k\geq 0}}
\left(\frac{x}{x+y} \right)^{k}
  (ky)^{n}\right)\frac{ t^{n}}{n!}.\nonumber
\end{align}
Equating $\frac{t^{n}}{n!}$ and by comparing with (\ref{highermomonts}), we obtain the result.
\end{proof}

\section{Degenerate generalized Fubini polynomials}
For any nonzero real number $\lambda$, we define the degenerate generalized Fubini polynomials as
\begin{align}
\frac{1}{1-\frac{x}{y}((1+\lambda ty)^{\frac{1}{\lambda}}-1)}
&  =
{\displaystyle\sum\limits_{n\geq 0}}
\mathfrak{F}_{n, \lambda} \left(  x, y\right)  \frac{t^{n}}{n!}.
 \label{genefuneneratedeg(x,y)2var}
\end{align}
It is clear that $\displaystyle\lim_{\lambda\rightarrow 0}(1+\lambda ty)^{\frac{1}{\lambda}}=e^{ty}\displaystyle$
and therefore $\displaystyle\lim_{\lambda \rightarrow 0}\mathfrak{F}_{n, \lambda} \left(  x, y\right)=\mathfrak{F}_{n} \left(  x, y\right)$.\par
Now, recall that the degenerate Stirling numbers of the second kind $\genfrac{\{}{\}}{0pt}{}{n}{k}_{\lambda}$, are defined by the following generating function \cite{Carlitz79}
\begin{align}
\frac{1}{k!}
\left(
(1+\lambda t)^{\frac{1}{\lambda}}-1
\right)^{k}
  &  =
{\displaystyle\sum\limits_{n\geq k}}
\genfrac{\{}{\}}{0pt}{}{n}{k}_{\lambda}  \frac{t^{n}}{n!}.
 \label{genefuneneratedegstirling}
\end{align}
In the next result, we will give the explicit formula for $\mathfrak{F}_{n,\lambda} \left(  x, y\right)$.
\begin{theorem} For $n\geq 0$, we have
\begin{align}
\mathfrak{F}_{n,\lambda} \left(  x, y\right)   = {\displaystyle\sum\limits_{k=0}^{n}}\genfrac{\{}{\}}{0pt}{}{n}{k}_{\lambda}
k!x^{k}y^{n-k}. \label{explicitFubinidegenerate}
\end{align}
\end{theorem}
\begin{proof}
From (\ref{genefuneneratedeg(x,y)2var}), we note that
\begin{align}
{\displaystyle\sum\limits_{n\geq 0}}
\mathfrak{F}_{n,\lambda} \left(  x, y\right)  \frac{t^{n}}{n!}   &  = \frac{1}{1-\frac{x}{y}((1+\lambda t y)^{\frac{1}{\lambda}}-1)}\\
&  ={\sum\limits_{k\geq0}}
\left(\frac{x}{y}\right)^{k}\left((1+\lambda t y)^{\frac{1}{\lambda}}-1\right)^{k}\nonumber\\
&={\sum\limits_{k\geq0}}
\left(\frac{x}{y}\right)^{k}k!
{\displaystyle\sum\limits_{n\geq k}}
\genfrac{\{}{\}}{0pt}{}{n}{k}_{\lambda}  \frac{\left(ty \right)^{n}}{n!}\nonumber\\
&={\sum\limits_{n\geq0}}\left( {\displaystyle\sum\limits_{k=0}^{n}}
\genfrac{\{}{\}}{0pt}{}{n}{k}_{\lambda}k! x^{k} y^{n-k}\right) \frac{t ^{n}}{n!}.\nonumber
\end{align}
Equating $\frac{t^{n}}{n!}$, we obtain the result.
\end{proof}
\begin{remark}
Now, by setting $y=1$ in \eqref{genefuneneratedeg(x,y)2var}, we get
\begin{align}
\frac{1}{1-x((1+\lambda t)^{\frac{1}{\lambda}}-1)}   &  =
{\displaystyle\sum\limits_{n\geq 0}}
\mathfrak{F}_{n,\lambda} \left(  x, 1\right)  \frac{t^{n}}{n!}\nonumber\\
&  ={\sum\limits_{n\geq0}}\omega_{n,\lambda}\left(  x\right)  \dfrac{t^{n}}{n!}\nonumber,\label{genefunFubbini}
\end{align}

where $\omega_{n,\lambda}\left(  x\right)  $ denotes the  Fubini polynomials \cite{kim2017note}, defined by%
\[
\omega_{n,\lambda}\left(  x\right)  :=
{\displaystyle\sum_{k=0}^{n}}
\genfrac{\{}{\}}{0pt}{}{n}{k}_{\lambda}k!
x^{k}.
\]
\end{remark}

\bibliographystyle{abbrv}
\bibliography{Mech}

\end{document}